\documentclass[runningheads,envcountsame]{llncs}
\usepackage[T1]{fontenc}
\usepackage[utf8]{inputenc}
\usepackage{graphicx}
\usepackage{lineno}

\usepackage{caption}
\usepackage{subcaption}

\RequirePackage[ruled,vlined,linesnumbered]{algorithm2e}

\SetCommentSty{mycommfont}
\SetKwInOut{Input}{Input}
\SetKwInOut{Output}{Output}
\SetKw{Continue}{continue}
\DontPrintSemicolon
\SetNoFillComment
\let\oldnl\nl
\newcommand{\nonl}{\renewcommand{\nl}{\let\nl\oldnl}}

\usepackage{amssymb}
\usepackage{mathtools}

\let\doendproof\endproof
\renewcommand\endproof{~\hfill$\qed$\doendproof}

\include{shortcuts}
\usepackage{listings}

\newif\ifpaper
\paperfalse 

\usepackage{todonotes}

\begin{document}
	\title{$k$-local Graphs}
	%
	%
	\author{Christian Beth\orcidID{0000-0003-3313-0752} \and
		Pamela Fleischmann\orcidID{0000-0002-1531-7970} \and
		Annika Huch\orcidID{0009-0005-1145-5806} \and
		Daniyal Kazempour\orcidID{0000-0002-2063-2756} \and
		Peer Kröger\orcidID{0000-0001-5646-3299}\and
		Andrea Kulow\and
		Matthias Renz\orcidID{0000-0002-2024-7700}
		}
	\authorrunning{C. Beth et al.}
	%
	\institute{Department of Computer Science, Kiel University, Germany
		\email{\{cbe,fpa,ahu,dka,pkr,mr\}@informatik.uni-kiel.de, stu229089@mail.uni-kiel.de}}
	\maketitle              

In 2017 Day et al. introduced the notion of locality as a structural complexity-measure for patterns in the field
of pattern matching established by Angluin in 1980.
In 2019 Casel et al. showed that determining the locality of an arbitrary pattern is NP-complete.
Inspired by hierarchical clustering, we extend the notion to coloured graphs, i.e., given a coloured graph determine an enumeration of the colours such that colouring the graph stepwise according to the enumeration leads to as few clusters as possible.
Besides initial theoretical results on graph classes, we propose a priority search algorithm to compute the $k$-locality of a graph.
This algorithm guarantees a correct result, while being faster by orders of magnitude than an exhaustive search over all permutations of colour enumerations.
Finally, we perform a case study on a DBLP subgraph to demonstrate the potential of $k$-locality for knowledge discovery.

	\section{Introduction}
	The notion of locality was firstly introduced by Day et al. as a structural complexity-measure for patterns
\cite{localpatterns}. A pattern is a word that consists of constants from an alphabet of terminal symbols and
variables. In their work, the authors investigated the \emph{matching problem} where one decides whether, given a
word $w$ and a pattern $\alpha$, the word may be obtained from substituting the variables in $\alpha$. While this
problem is $\NP$-complete in general \cite{DBLP:journals/jcss/Angluin80}, they showed that the matching and also 
the membership problem can be solved efficiently for the class of $k$-local patterns for any fixed $k$.
Extending the idea of the matching algorithm for \emph{non-cross patterns} \cite{shinohara1982polynomial}, 
for $k$-local patterns one has at every given point during the substitution at most $k$ separate factors within the pattern, i.e., the algorithm in \cite{localpatterns} tracks the possible extensions of such factors via dynamic programming. In \cite{cutwidth} the research on $k$-locality is continued from a combinatorial perspective. 
The authors investigated the problem of computing the locality number of words and showed via a reduction from the decision problem $\mathsf{cutwidth}$ on graphs (asking to decide whether the cutwidth
of a graph is upper bounded by a given number) that the computation of the locality number is $\NP$-complete.
The studies of $k$-local words from a word combinatorial perspective continued in \cite{fleischmann}. 
A first connection between $k$-locality and graphs was introduced in \cite{SofsemTim}, where the authors characterised the graphs which are represented by $1$-local words.


Hierarchical clustering is an archetype of clustering methods that allow for partitioning a given dataset.
It operates on distances between objects (instances) and depending on the used linkage distance (e.g., single or complete link) and the clustering strategy like agglomerative (bottom-up) or divisive (top-down) a hierarchy of partitions among all objects is computed.
Unlike other archetypes like, e.g., density-based or variance-minimising approaches, hierarchical clustering allows one to see and understand either single objects or clusters within their local context (e.g., a cluster A and B are most similar to each other within a cluster C).
To combine hierarchical clustering \cite{sibson1973slink} and the notion of locality on graphs, we investigate coloured graphs and enumerations of the used colours.
In the first step, all vertices of the first colour in the enumeration form a singleton cluster each.
In an arbitrary step, the vertices of the actual colour are either extending an existing cluster (if there is an edge from the vertex to a vertex in the cluster) or they form a new singleton cluster (otherwise).
See Figure~\ref{klocalexample} for an example.

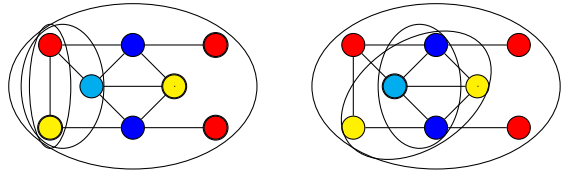
\begin{figure}
    \centering
    \begin{tikzpicture}[scale=0.55]
        \node[shape=circle,draw=black,fill=yellow] (A) at (0,0) {};
        \node[shape=circle,draw=black,fill=red] (B) at (0,2) {};
        \node[shape=circle,draw=black,fill=blue] (C) at (2,2) {};
        \node[shape=circle,draw=black,fill=cyan] (D) at (1,1) {};
        \node[shape=circle,draw=black,fill=blue] (E) at (2,0) {};
        \node[shape=circle,draw=black,fill=yellow] (F) at (3,1) {} ;
        \node[shape=circle,draw=black,fill=red] (G) at (4,0) {} ;
        \node[shape=circle,draw=black,fill=red] (H) at (4,2) {} ;
        
        \path [-] (A) edge node[left] {} (B);
        \path [-](B) edge node[left] {} (C);
        \path [-](B) edge node[left] {} (D);
        \path [-](C) edge node[left] {} (H);
        \path [-](C) edge node[left] {} (F);
        \path [-](D) edge node[left] {} (C);
        \path [-](D) edge node[left] {} (F);
        \path [-](D) edge node[left] {} (E);
        \path [-](A) edge node[right] {} (E);
        \path [-](E) edge node[right] {} (F);   
        \path [-](E) edge node[right] {} (G);  
        \path [-](F) edge node[right] {} (F);   
        
        \draw (A) circle (0.3);
        \draw (F) circle (0.3);
        \draw (G) circle (0.3);
        \draw (H) circle (0.3);
        
        \draw (0,1) ellipse (0.5cm and 1.5cm);
        
        \draw (0.3,1) ellipse (1cm and 1.5cm);
        
        \draw (2,1) ellipse (3cm and 2cm);
    \end{tikzpicture}
    \hspace{0.5cm}
    \begin{tikzpicture}[scale=0.55]
        \node[shape=circle,draw=black,fill=yellow] (A) at (0,0) {};
        \node[shape=circle,draw=black,fill=red] (B) at (0,2) {};
        \node[shape=circle,draw=black,fill=blue] (C) at (2,2) {};
        \node[shape=circle,draw=black,fill=cyan] (D) at (1,1) {};
        \node[shape=circle,draw=black,fill=blue] (E) at (2,0) {};
        \node[shape=circle,draw=black,fill=yellow] (F) at (3,1) {} ;
        \node[shape=circle,draw=black,fill=red] (G) at (4,0) {} ;
        \node[shape=circle,draw=black,fill=red] (H) at (4,2) {} ;
        
        \path [-] (A) edge node[left] {} (B);
        \path [-](B) edge node[left] {} (C);
        \path [-](B) edge node[left] {} (D);
        \path [-](C) edge node[left] {} (H);
        \path [-](C) edge node[left] {} (F);
        \path [-](D) edge node[left] {} (C);
        \path [-](D) edge node[left] {} (F);
        \path [-](D) edge node[left] {} (E);
        \path [-](A) edge node[right] {} (E);
        \path [-](E) edge node[right] {} (F);   
        \path [-](E) edge node[right] {} (G);  
        \path [-](F) edge node[right] {} (F);   
        
        \draw (D) circle (0.3);
        
        \draw (1.6,1) ellipse (1cm and 1.5cm);
        
        \draw[rotate=-56] (0.2,1.7) ellipse (1.3cm and 2cm);
        
        \draw (2,1) ellipse (3cm and 2cm);
    \end{tikzpicture}
    \caption{A $4$-coloured graph $G$ marked with the marking sequences $(\mathtt{yellow}, \mathtt{red}, \mathtt{cyan}, \mathtt{blue})$ and $(\mathtt{cyan}, \mathtt{blue}, \mathtt{yellow}, \mathtt{red})$.}
    \label{klocalexample}
\end{figure}

Since all words (strings, subsequences) can be interpreted as linear graphs, determining the locality of a graph remains $\NP$-complete (cf. \cite{cutwidth}). But similar to the locality of sort-of palindromic structured words \cite{localpatterns},
some graph classes and some graphs with a specific colouring yield an efficiently determinable locality.

Besides theoretical interest, the locality of graphs bears a variety of potential relationships to the machine learning and data science community.
In co-location pattern mining, a \textit{co-located pattern} refers to instances (objects) of different types (subsets of features) frequently occurring together at the same (or at a close) location \cite{Mamoulis2008,shekhar2001discovering,morimoto2001mining}.
Applications for co-location pattern mining are co-occurrences of buildings/companies/institutions \cite{masrur2019co}, social events, and biological and societal phenomena in temporal context.
In this context, one is interested in the largest locality a colour enumeration may provide.
Another related data mining task is motif discovery in gene or protein sequences. 
While the locality of words is researched as a structural complexity measure \cite{fleischmann}, this measure can contribute to finding motifs (frequently re-occurring patterns) and clustering of similar motifs.
More explicitly, according to \cite{yu2022method}, determining the length of motifs is still an open challenge and the authors of \cite{yu2022method} propose a deep learning based approach to determine the length.
The locality of words can provide here the means to estimate that motif size $\vert M \vert_{\mu}$.
Knowing the maximum number of blocks and the total sequence length $\vert S \vert$, the average motif length would be $\vert M \vert_{\mu} = \frac{\vert S \vert}{k}$.
Being aware that this is not an exact value, it can serve as an orientation for domain experts \textit{around} which value the motif length parameter can be chosen from.
This orientation for parameter selection is common, e.g., the silhouette score as a cluster quality metric \cite{rousseeuw1987silhouettes}, which informs a meaningful choice of cluster numbers for $k$-means and other variance minimising clustering algorithms.
However, not only the discovery of motifs alone requires the length of the motifs, but also their clustering \cite{hamady2008motifcluster} requires both their length and number.
All these applications require a feasible solution to compute the $k$-locality.
However, since an exhaustive search over all possible marking sequences would be infeasible, a more scalable solution is required.

\textbf{Our Contribution.} In this work, we extend the notion of $k$-locality to graphs.
We consider (not necessarily valid or minimally) coloured, undirected graphs and enumerations of the colours.
Similar as for the $k$-locality of words, we count the number of connected components (instead of blocks) within the graph after each marking step.
We present first theoretical results regarding the complexity and special graph classes.
Further, we propose a priority search algorithm that is optimal in the number of marking prefix expansions to compute the $k$-locality of graphs.
We demonstrate its efficiency on scale-free graphs and conduct a case study on the DBLP publication graph to show some of its potential for knowledge discovery.

\textbf{Structure of the Work.} 
In Section~\ref{prelims} all required definitions and notions are introduced. Afterwards our theoretical results and practical applications follow in Sections~\ref{theo} and \ref{pract}. In Section~\ref{conc} we give a perspective of future work.


	\section{Preliminaries}\label{prelims}
	Let  $\N = \{1,2,\ldots\}$ denote the natural numbers and set $\N_0 = \N
\cup \{0\}$ as well as $[n]=\{1,\ldots,n\}$ and $[i,n]=\{i, i+1, \ldots, n\}$ for all $i,n\in\N_0$ with $i \leq n$.
Set $\N_{\geq k}=\N\backslash[k-1]$ for a natural number $k\geq 2$. 
For $q,n,r,z$ with $q \cdot n+r = z$, we denote the division without remainder by $q = z \operatorname{div} n$ and the rest by $r = z \operatorname{mod} n$.
Before we can introduce the new notion of $k$-local graphs, we need some basic definitions from combinatorics on words and from graph theory.

An \emph{alphabet} $\Sigma$ is a finite set of symbols, called
\emph{letters}.  A \emph{word} $w$ is a finite sequence of letters from a given alphabet and its length $|w|$ is the number of $w$'s letters. For $i \in
[|w|]$ let $w[i]$ denote $w$'s \nth{$i$} letter.  The set of all
finite words (strings, sequences) over the alphabet $\Sigma$, is denoted by $\Sigma^{\ast}$. The empty word
$\varepsilon$ is the word of length $0$. Let $\Sigma^n$ denote all words in $\Sigma^{\ast}$ exactly of length $n\in\N_0$.
Set $\letters(w) = \{\ta \in \Sigma \mid \exists i \in [|w|]: w[i] = \ta \}$ as $w$'s alphabet. 
For $u,w\in\Sigma^{\ast}$, $u$ is called a \emph{factor}
of $w$, if $w = xuy$ for some words $x,y\in\Sigma^{\ast}$.  For $1\leq i\leq j\leq|w|$ 
define the factor from $w$'s \nth{$i$} letter to the \nth{$j$} letter by  $w[i,j]=w[i]\cdots w[j]$.
Further, given a pair of indices $i  < j$, set $w[j, i] = \varepsilon$.

A \emph{finite, undirected graph} $G$ is a pair $G=(V,E)$ with finite sets $V$ of vertices and edges $E \subseteq \{M\subseteq V\mid |M|=2\}$ (the latter set is abbreviated by $\binom{V}{2}$).
A \emph{path} in $G$ is a finite sequence of vertices $(v_1, \dots, v_k)$ with $\{v_i,v_{i+1}\}\in E$ for all $i\in[k-1]$ and $k \in \N_{\geq 2 }$. We say that this path has length $k-1$. Further, two vertices are \emph{connected} if there exists a path between them.
By convention set $n=|V|$ and $m=|E|$ and since we are only investigating finite, undirected graphs, we call them simply {\em graphs}. Moreover, given a graph $G$, let $V(G)$ and $E(G)$ denote its sets of vertices and edges resp.
For $v \in V$, define the \emph{connected component} of $v$ by
$C(v)=\{u\in V \mid \exists \mbox{ path from } u \mbox{ to } v \}$.
Let $\gamma(G)$ denote the number of connected components in $G$.
We say that $G$ is \emph{connected}, if $V = C(v)$ for all $v \in V$, i.e., $\gamma(G)=1$.
For $n \in \N$, the \emph{complete graph} on $|V| = n$ vertices is defined by $K_n = (V, \binom{V}{2})$. 
A graph $H = (U,F)$  is called a \emph{subgraph} of $G$ ($H \leq G$) if we have $U\subseteq V$ and $F\subseteq E\cap\binom{U}{2}$.	Moreover, we say that $H\leq G$ is \emph{induced} by $U$ if $F = E\cap \binom{U}{2}$; this is notated as $G|_{U}$. A fully connected subgraph of $G$ is called a \emph{clique}.
For $\ell \in \mathbb{N}$, a \emph{valid $\ell$-colouring} of $G$ is a function $c : V \to [\ell]$ such that $\{u,v\} \in E$ implies $c(u) \neq c(v)$ for all $u, v \in V$.
	If there exists a valid $\ell$-colouring for $G$, we say that $G$ is $\ell$-colourable. 
	Let $\mathcal C (G,c) = \{ c(v) \mid v \in V\}$ denote the set of colours that occur in $G$ w.r.t.~the colouring $c$.
	For $x \in \mathcal C (G,c)$ and a valid $\ell$-colouring $c$ with $\ell \in \mathbb{N}$, denote the \emph{set of all $x$-coloured vertices} by $V_x = \{v \in V \mid c(v) = x \}$. Further, let $G|_C = G|_{\bigcup_{x \in C} V_{x}}$ for $C \subseteq \mathcal C (G,c)$ be the {\em subgraph of $G$ induced by the colours in $C$}.

Now, we define the $k$-locality on graphs 
generalising the definitions from words (cf. \cite{localpatterns}). See Appendix~\ref{furtherdefs} for more basics on $k$-local words and several of the investigated graph classes. From now on let $G=(V,E)$ be an undirected graph.

\begin{definition}
	Let $\overline{V} = \{\overline{x} \mid x \in V\}$ be a copy of $V$ that we refer to as the {\em set of marked vertices} and $ c: V \rightarrow [\ell]$ with $\ell \in \N$ an $\ell$-colouring of $G$.
	For $G$ and $c$, we define a {\em marking sequence} $e$ of the colours occurring in $G$ as an enumeration $(x_1, x_2, \ldots, x_{|\mathcal C(G,c)|})$ of $\mathcal C (G,c)$ such that $x_i \leq_e x_j$, if $x_i$ is enumerated before $x_j$ in $e$ for $i, j \in [|\mathcal C (G,c)|]$.
	A vertex $v$ is marked in marking step $j \in \N$ if $c(v) = x_i$ for some $i \leq j$. Set $G_i = G|_{\{x_1,\ldots,x_i\}}$ for all $i \in [\ell]$, i.e., $G_i$ is the graph induced by the first $i$ colours.
\end{definition}

\begin{definition}\label{graphloc1}
	Let $e = (x_1, \ldots, x_{|\mathcal C (G,c)|})$ be a marking sequence of $\mathcal C (G,c)$.
	Define $\loc(G,c,e) = \max_{x_i \in e} |C(G_i)|$,
	i.e., the {\em locality of  $G$ w.r.t.~$e$} is given by the maximal number of connected components occurring during the marking process. Define the \emph{locality of $G$} by $\loc(G,c) = \min_{e} \loc(G,c,e)$.
\end{definition}

\begin{definition}\label{defgraphloc}
	A graph $G$ with an $\ell$-colouring $c$ is $k${\em -local} for $k,\ell \in \N$, if there exists a marking sequence $(x_1, \ldots, x_{|\mathcal C (G,c)|})$ of $\mathcal C (G,c))$ such that for all $i \leq |\mathcal C (G,c)|$,  $G_i$ contains at most $k$ connected components in marking step $i$. 
\end{definition}

\begin{definition}
	A graph $G$ is \emph{strictly $k$-local} if $G$ is $k$-local but not $(k-1)$-local.
\end{definition}

Notice that in general, we are not only interested in valid or minimal colourings. In the next section, we will show that a restriction to valid colourings is w.l.o.g. possible. We finish this section by revisiting the example from 	Figure~\ref{klocalexample}.  Consider the marking sequence $e= (\mathtt{yellow}, \mathtt{red}, \mathtt{cyan}, \mathtt{blue})$. When marking all $\mathtt{yellow}$ vertices, we obtain $2$ connected components. Then marking $\mathtt{red}$ we result in $4$ connected components. Marking $\mathtt{cyan}$ we still have $4$ components and then $\mathtt{blue}$ both result in one component. By \Cref{graphloc1} we have $\loc(G,c,e) = 4$. One can verify that we have $\loc(G,c,e')=1$ for the marking sequence $e' = (\mathtt{cyan}, \mathtt{blue}, \mathtt{yellow}, \mathtt{red})$. Thus, $\loc(G,c) = 1$.

 The computational model we use is the standard unit-cost RAM with logarithmic word size (cf. e.g.,~\cite{crochemore}).

	\section{Theoretical Results: Complexity and Graph Classes}\label{theo}
	We start this section with a reduction from the $\NP$-complete  problem $\Loc_w$ \cite{cutwidth} to the locality in the graph setting. First, we define the locality problem on graphs formally and obtain two immediate results from the locality on words.

\begin{problem}
	Let $\Loc_G$ be the problem of deciding whether for a connected graph with a valid $\ell$-colouring $c$, we have $\loc(G,c) \leq k$. The subproblem to decide whether we have $\loc(G,c) = 1$ is abbreviated by $\Loc_{G,1}$.
\end{problem}

\ifpaper 
\else    
\begin{proposition}\label{membership}
	Given $k \in \N$, a graph $G$, a valid $\ell$-colouring $c$ and a marking sequence $e = (c_1, \ldots, c_\ell)$, we can verify whether $\loc(G,c) \leq k$ in polynomial time.
\end{proposition}
\begin{proof}
	In every marking step $i \in [\ell]$, we need to compute the (strongly) connected components of the induced subgraph $G|_{\{x_1, \ldots, x_i\}}$. Well-known results on strongly connected components allow for a linear time computation of those (strongly) connected components. 
\end{proof}
\fi

\begin{proposition}\label{completeness}
	For an undirected graph $G$, the problem $\Loc_G$ is $\NP$-complete.
\end{proposition}
\ifpaper 
\else    
\begin{proof}
	The $\NP$-membership follows from Proposition~\ref{membership}. Thus, it is left to show the $\NP$-hardness. First, consider the case that $(u',k) \in \Loc_w$. Now consider $u = \cond(u')$ and note that $\loc(u) = \loc(u')$. Thus, there exists a marking sequence $e = (x_1, \ldots, x_{|\Sigma|}) \in \Sigma^{|\Sigma|}$ such that $u$ is marked with at most $k$ blocks. First, we construct a graph $G = (\{u[i]_i \mid i \in [|u|]\}, \{\{u[i]_i,,u[i+1]_i \} \mid i \in [|u|-1]\})$ and a valid $|\al(u)|$-colouring $c:V \to \al(u)$ by $y_i \mapsto y$ for $i \in [|u|]$ (note that this is a valid colouring since $u$ is in condensed form). Now we are showing that $\loc(G,c) \leq k$ holds. 
	Let $(b_1, \ldots, b_{|\Sigma|})$ be the sequence of the number of blocks in each marking step from the marking sequence $e$. By the construction of the graph, we can conclude that for all $i \in [|\Sigma|]$ we have $|C(G|_\{x_1,\ldots, x_i\}) = b_i$ for $i \in [|\Sigma|]$. Since $b_i \leq k$ for all $i \in [|\Sigma|]$, we have that $\loc(G,c,e) = \max_{x_i \in e} |C(G|_{\{x_1, ..., x_i\}})| = \max_{i \in |\Sigma|} b_i \leq k$. This $(G,c,k) \in \Loc_G$.
	
	In the case $(u',k) \notin \Loc_w$ we obtain $\loc_{G,c,e} = \max_{x_i \in e} |C(G|_{\{x_1, ..., x_i\}})| = \max_{i \in |\Sigma|} b_i > k$ following the same line of arguments as above. Thus, $(G,c,k) \notin \Loc_G$.
\end{proof}
\fi

\begin{theorem}
 For every undirected graph $G$, we have $\Loc_{G,1}\in\mathsf{P}$.
\end{theorem}
\begin{proof}
	Let us firstly sketch a nondeterministic algorithm that decides this
	problem: Guess the first colour that occurs in the graph, and mark
	all vertices in the graph. Now as long as possible pick a colour for
	which every vertex is connected to a marked vertex. If there is no
	vertex left in the graph then stop and accept the input.
	
	The algorithms finds only valid colour
	orderings. However, it remains to show that the algorithm always
	finds such an ordering if it exists. Assume there is a valid colour
	ordering, then the algorithm would non-deterministically guess the
	first colour correctly. Now assume that our algorithm would not
	accept. This can only happen when all remaining colours are connected
	to at least on not marked vertex. Consider all colours of the valid
	ordering, there is a first colour that has not been marked yet. The
	valid ordering guarantees that the number of connected components
	would not exceed 1. Indeed there are already more vertices marked,
	however since these are all connected to the first colour they would
	not increase the number of components. 
	
	Since we only have a constant number of guesses and a polynomial
	number of possible vertices, the algorithm runs deterministically in
	polynomial time.
\end{proof}

Note that we cannot restrict the investigations of $k$-locality w.l.o.g.~to connected graphs since the locality of a graph does not depend (to our knowledge) just on the locality of the connected components. 	Consider the two graphs $G$ and $G'$ from Figure~\ref{connectedexample}. One can verify with the respective marking sequences that all the connected components of $G$ and $G'$
each are $1$-local. But both possible marking sequences for $G$ lead to $3$-locality. Also in $G'$ both connected components
are $1$-local; $G'$ itself is $2$-local witnessed by  $(\mathtt{red}, \mathtt{yellow},\mathtt{blue})$.
	\begin{figure}
		\centering
		\begin{tikzpicture}[scale=0.7]
			\node (g) at (-0.5,0) {$G$};
			\node[shape=circle,draw=black,fill=yellow] (A) at (0,0) {};
			\node[shape=circle,draw=black,fill=red] (B) at (1,0) {};
			\node[shape=circle,draw=black,fill=yellow] (C) at (2,0) {};
			\node[shape=circle,draw=black,fill=red] (D) at (0,-1) {};
			\node[shape=circle,draw=black,fill=yellow] (E) at (1,-1) {};
			\node[shape=circle,draw=black,fill=red] (F) at (2,-1) {} ;
			
			\node (g') at (3.5,0) {$G'$};
			\node[shape=circle,draw=black,fill=yellow] (G) at (4,0) {} ;
			\node[shape=circle,draw=black,fill=red] (I) at (5,0) {} ;
			\node[shape=circle,draw=black,fill=yellow] (J) at (6,0) {} ;
			\node[shape=circle,draw=black,fill=blue] (K) at (7,0) {} ;
			\node[shape=circle,draw=black,fill=blue] (L) at (4,-1) {} ;
			\node[shape=circle,draw=black,fill=red] (M) at (5,-1) {} ;
			\node[shape=circle,draw=black,fill=blue] (N) at (6,-1) {} ;
			
			\path [-] (A) edge node[left] {} (B);
			\path [-](B) edge node[left] {} (C);
			
			\path [-](E) edge node[left] {} (F);
			\path [-](D) edge node[left] {} (E);

			\path [-](G) edge node[right] {} (I);   
			\path [-](I) edge node[right] {} (J);  
			\path [-](J) edge node[right] {} (K);
			
		 	\path [-](L) edge node[right] {} (M);
		 	\path [-](M) edge node[right] {} (N);
		\end{tikzpicture}
		\caption{A $2$-coloured graph $G$ and a $3$-coloured graph $G'$.}
		\label{connectedexample}
	\end{figure}
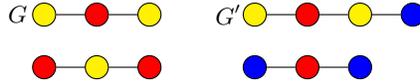
Note that for any coloured graph~$G$ with colouring $c$, we have $\loc(G,c)\geq\gamma(G)$.

\begin{proposition}\label{prop8}
	For every graph $G=(V,E)$ there is a (valid) colouring~$c$ such that $\loc(G,c)=\gamma(G)$.
\end{proposition}
\begin{proof}
	
	If $c$ is restricted to be a valid colouring, then consider a
	colouring where every colour appears only for one vertex. We can find
	a permutation that certifies the claimed locality. Just iterate over
	all connected components, and each time pick an arbitrary vertex.
	Then, the induced graph has exactly $\gamma(G)$ components/single
	vertices. Then, greedily pick remaining vertices that are in the
	neighbourhood of the already chosen ones.
\end{proof}

\begin{remark}\label{everygraph1local}
	Note that Proposition~\ref{prop8} implies that there always exists a (not necessarily minimal) valid colouring for any connected graph $G$ such that $G$ is $1$-local for this colouring.
\end{remark}

The more general version with non-connected graphs is at least as hard as the more
restricted version because computing components can be done in
polynomial time. 
In this work, we only investigate connected graphs. The following proposition justifies why we mostly consider validly coloured graphs. The idea is to contract adjacent vertices coloured with the same colour. 

\begin{proposition}
	For every graph $G$ with colouring $c$ there exists a validly coloured (by $c$) induced minor $G'$ of $G$ such that
	$\loc(G,c)=\loc(G',c)$.
\end{proposition}
\begin{proof}
	Consider  an edge $\{v,w\}\in E$ with
	$c(v)=c(w)$. For every edge $\{w,u\}\in E$ where $u\neq v$ replace
	$\{w,u\}$ by $\{v,u\}$, and finally remove $u$ and $\{v,w\}$ from
	$G$. Repeat this step, until no such an edge exists. Thus, build the minor induced by contracting adjacent vertices coloured with the same colour. This graph is moreover validly coloured. Since the colours occurring in the neighbourhood of a vertex remains the same, the locality does not change.
\end{proof}

A last result on arbitrary graphs narrows the locality a graph may have.

\begin{theorem}
	For all graphs $G$ coloured by $c$, we have $\gamma(G)\leq \loc(G,c)\leq \alpha(G)$ where $\alpha(G)$ denotes the size of a maximal independent set.
\end{theorem}
\begin{proof}
The lower bound follows immediately from the fact that for vertices of different connected components are not connected.
For the upper bound consider a marking sequence
	$(x_1,\dots,x_{\ell})$ satisfying $\loc(G,c)=k$. Thus, there exists $t\leq \ell$ such that
	\[
	\gamma(G|_{\bigcup_{j=1}^{t} V_{c_j}})=\loc(G,c),
	\]
	i.e., at marking stage $t$, $k$ connected components are reached. Pick a set $I$ of $k$
	many vertices that are pairwise not connected in $G$ -  one vertex of each component. Hence, $I$ is
	an independent set of size $\loc(G,c)$, so the maximum independent
	set in $G$ has size at least $\loc(G,c)$, which completes the proof.
\end{proof}

The remainder of this section dedicated to several graph classes and is meant a beginning of this line of research.  We start with cycles, stars, complete, wheel, web and friendship graphs (cf. \Cref{tablegraphs}).
Due to Remark~\ref{everygraph1local} we focus on validly (not necessarily minimally) coloured graphs.

\begin{remark}
The triangle $C_3$ is $1$-local for every colouring. For $n\geq 4$, the cycle graph $C_n$ is $2$-colourable if $n$ is even and $3$-colourable otherwise. Thus,  minimal coloured cycle graphs are $\min\{|V_1|,|V_2|\}$-local for the sets  of vertices $V_1,V_2$ containing at least two nodes.
\end{remark}

\begin{proposition}\label{completestar}
	Complete, star, wheel and friendship graphs are $1$-local independent of the colouring.
\end{proposition}
\begin{proof}
	Let $c$ be an arbitrary colouring of a complete graph $K_n = (V,E)$ on $n$ vertices. Since $|V_c| = 1$ holds for all $c \in \bigcup_{v \in V} c(v)$ the claim follows.
	
	In a star we can pick the colour of the vertex that has more than one
	neighbour first. Since, there is the only one such vertex, which is
	also connected to all other vertices, we have at most one component
	after this step. Now, we can pick the remaining colours in an
	arbitrary ordering. The number of components remains one, because
	every other vertex is connected with the first vertex.
	
	By the same argument as for star graphs, we can start with the central vertex for wheel and friendship graphs, too. Now, every other vertex is connected to the marked one. Thus, the remaining colours can be picked again in an arbitrary order.
\end{proof}

\begin{proposition}
	If $\ell \in \N_{\geq 2}$ is even, web graphs are $\frac{\ell r}{2}$-local for every valid $2$-colouring. If $\ell$ is odd, there exists a $3$-colouring such that $W_{\ell,r}$ is $(\ell \operatorname{div} 3) r$-local.
\end{proposition}
\ifpaper 
\else    
\begin{proof}
	For even $\ell$, the $\frac{\ell r}{2}$-locality follows directly from the only possibility to validly $2$-colour $W_{\ell,r}$. There are $\ell r$ vertices in total where half of them have the same colour. For odd $\ell$, we are going to define a $3$-colouring and show its  $(\ell \operatorname{div} 3) r$-locality. Consider a web graph consisting of $r$ $\ell$-cycles $u_1^1 \cdots u_1^\ell$, $\ldots, u_r^1\cdots u_r^\ell$ and the resp. connections $u_1^i \cdots u_r^i$ for $i \in [\ell]$. We define the colouring function $c$ for $\mathtt{r},\mathtt{g},\mathtt{b} \in \N$, $i \in [\ell], j \in [r]$ as follows:
	\begin{align*}
		u_j^i \mapsto c(v) = \begin{cases}
			\mathtt{b}, & \text{if } j \text{ odd } \land i=1,\\
			\mathtt{r}, &  \text{if } j \text{ odd } \land i \text{ even},\\
			\mathtt{g}, &  \text{if } j \text{ odd } \land i\neq 1 \land i \text{ odd},\\
			\mathtt{b}, & \text{if } j \text{ even } \land i=\ell,\\
			\mathtt{r}, &  \text{if } j \text{ even } \land i\neq \ell \land i \text{ odd},\\
			\mathtt{g}, &  \text{if } j \text{ even } \land i \text{ even.}
		\end{cases}
	\end{align*}
	So roughly speaking, the colouring of the inner circle is the same as for every copy of the circle in an odd layer and rotated once for any copy of an even layer.
	By definition of $c$, there exists a colour $x\in \{\mathtt{r},\mathtt{g}\}$ such that $(\ell \operatorname{div} 3)r$ vertices are $x$-coloured. Let $y$ denote the other colour. This colour is the first in the marking sequence we define by $e = (x,y,\mathtt{b})$. By marking $x$, we obtain $(\ell \operatorname{div} 3)r$ connected components. By construction, every $x$-coloured vertex is connected to a $y$-coloured one which shows the claim.
\end{proof}

\fi

\begin{proposition}
	For a minimal valid colouring, sunflower graphs $S_d$ and helm graphs $H_d$ for $d\in \N$ are $\lfloor \frac{d-1}{2}\rfloor$-local.
\end{proposition}
\ifpaper 
\else    
\begin{proof}
Since sunflower and helm graphs consist of a cycle, we can conclude that a minimal colouring needs to consist of at least $2$ colours if $(d-1)$ is even and $3$ colours of $(d-1)$ is odd. Let those colours be denoted by $r,g,b\in \N$ where $b$ is the colour only used if $(d-1)$ is odd. Furthermore, the inner vertex is connected to all vertices of the cycle and thus needs to be coloured in another colour $y\in \N$. The same colour as for the inner vertex can be used for the outer cycles $C_3$. Thus, the marking sequence $e=(b,r,g,y)$ results in a $\lfloor(d-1)/2\rfloor$-locality of $S_d$ since there exist $(d-1)+1=d$ vertexes coloured in $y$, so marking $y$ before $r$ and $g$ leads to a higher locality. 
\end{proof}
\fi

We now focus on the class of bipartite graphs for which we can determine the locality for a $2$-colouring by inspecting the number of vertices coloured in each of the two colours.

\begin{proposition}
	Let $G=(V_1\dot{\cup} V_2,E)$ be a bipartite, graph that is $2$-coloured. Then $G$ is strictly $\min\{|V_1|,|V_2|\}$-local.
\end{proposition}
\begin{proof}
	Let $x_1,x_2$ be two colours. A valid 2-colouring of the bipartite graph $G=(V_1\dot{\cup} V_2,E)$ is given by $c : V_1\dot{\cup} V_2 \to \{x_1,x_2\}$ with $c(v) = x_i$ for $v\in V_i$, $i\in[2]$. Let w.l.o.g. $|V_1| \leq |V_2|$ then the marking sequence $e = (x_1,x_2)$ shows the $|V_1|$-locality of $G$.
\end{proof}

\begin{corollary}
	For $d\in\N$, hypercubes $Q_d$ are strictly $2^{d-1}$-local, every knight's graph on $KG = (V_1\dot{\cup}V_2,E)$ and gear graphs $G_d$ are strictly $\min\{|V_1|,|V_2|\}$-local and $d$-crown graphs $(\{x_1, \cdots, x_d, y_1, \ldots, y_d\},\{(x_i,y_j) \mid i,j \in [d], i\neq j\})$ are strictly $d$-local for $d \in \N_{\geq 3}$.
\end{corollary}
\ifpaper 
\else    
\begin{proof}
	The claim follows from the bipartiteness of all graph classes. For hypercubes note that they have $2^d$ vertices and the colours occur equally often in any $2$-colouring. i.e., their locality is given by $\frac{2^d}{2} = 2^{d-1}$.
\end{proof}

\fi


\begin{corollary}
	For a complete bipartite graph we have $\loc(G,c) = \min\{|V_1|,|V_2|\}$ for any valid colouring $c$.
\end{corollary}

\begin{corollary}
	For $d\geq 2$, every complete $d$-partite graph  $K_{i_1, \ldots, i_d}$ for $i_1,\ldots, i_d\in \N$ is $\min\{i_1, \ldots, i_d\}$-local where $i_j$ denotes the number of vertices in every disjoint vertex set, $j \in [d]$.
\end{corollary}

We now inspect  bipartite graphs that are coloured with more than $2$ colours.

\begin{proposition}
	Let $G=(V_1\dot{\cup} V_2,E)$ be a bipartite graph that is $(>2)$-coloured by a colouring $c$. If there exists a colour $x\in \mathcal{C}(G,c)$ such that for all $v \in V_i$ we have $c(v) = x$ for exactly one $i \in [2]$ and $|V_i| \leq |V_j|$ for $j \in [2] \setminus \{i\}$ then $G$ is strictly $|V_i|$-local. Otherwise $G$ is strictly $|V_j|$-local.
\end{proposition}
\begin{proof}
	Let w.l.o.g. $V_1$ the set of nodes such that there exists a colour $x$ with $c(v) = x$ for all $v \in V_1$ and $|V_1| \leq |V_2|$. Then for any marking sequence that starts with $x$ the graph is $|V_1|$-local. Further, any marking sequence $\sigma'$ that does not start with $c$ but $c'$ there exists a set of nodes $V_{c'}$ such that there are $|V_{c'}|$ connected components. Suppose that $x$ is the last colour of the marking sequence $\sigma'$. Then the number of connected components after $|\sigma'| -1$ marking steps is $|V_2|$, but $|V_2| \geq |V_1|$ which implies that the locality of $G$ w.r.t. $\sigma'$ is larger than the one w.r.t. $\sigma$, i.e., $G$ is $|V_1|$-local. In the case of a marking sequence $\sigma''$ such that there exists $i \in [2, \sigma -1]$ with $\sigma[i] = x$, i.e., we mark the $x$-coloured nodes somewhen in the middle, the number of connected components is given by $\max\{|V_1|, |S|\}$ where $S\subseteq V_2$ is given by $\bigcup_{j \in [i-1]} \{v \in V_2 \mid c(v) = x_j\}$. If $|V_1| \geq |S|$ then the locality of $G$ is given by $|V_1|$. Otherwise, the marking sequence $\sigma'$ is not an optimal one. 
	
	In the case that $|V_1| < |V_2|$ a marking sequence that firstly marks all nodes of $V_2$ results in a locality of $|V_2|$.
\end{proof}

The previous results raise the hope that for all bipartite graphs and all colourings, we can determine the locality efficiently.
\Cref{biparbitrary} shows that this is not the case.

\begin{proposition}\label{tildewconstruction}
	Let $\tilde{\Sigma} = \{\widetilde{\ta}_i \mid \ta_i \in \Sigma\}$. For all condensed $w \in \Sigma^*$ there exists $\tilde{w} \in (\Sigma \cup \tilde{\Sigma})^*$ such that for all $x \in \letters(w)$ we have either $\{i \mid w[i] = x\} \subset 2\mathbb{N}$ or $\{i \mid w[i] = x\} \subset 2\mathbb{N}+1$. Further, $w$ is strictly $k$-local iff $\tilde{w}$ is strictly $k$-local.
\end{proposition}
\ifpaper 
\else    
\begin{proof}
	Let $\Sigma = \{\ta_1, \ldots, \ta_\sigma\}$. Define a mapping 
	\[
	f_n : \{\ta_1, \ldots, \ta_\sigma\}^n \to (\{\ta_1, \ldots, \ta_\sigma\} \cup \{\widetilde{\ta_1}, \ldots, \widetilde{\ta_\sigma}\})^*
	\]
	for $n \in \mathbb N$ by
	\[
	w = w[1] \cdots w[n] \mapsto f_n (w) = \tilde{w} = w[1] \widetilde{w[2]}^{i_1} w[2] \widetilde{w[3]}^{i_2} w[3] \widetilde{w[4]}^{i_3} \cdots \widetilde{w[n]}^{i_{n-1}} w[n] 
	\]
	with $i_1, \ldots, i_{n-1} \in \{0,1\}$ such that for all $x \in \{\ta_1, \ldots, \ta_\sigma\} \cup \{\widetilde{\ta_1}, \ldots, \widetilde{\ta_\sigma}\}$ we have either $\{i \in [n] \mid w[i] = x\} \subset 2 \mathbb N$ or $\{i \in [n] \mid w[i] = x\} \subset 2 \mathbb N +1$. In the following we will refer to letters from $\{\widetilde{\ta_1}, \ldots, \widetilde{\ta_\sigma}\}$ as \emph{shifting letters}.
	
	Note that $f_n$ is well-defined since every letter from $\{\ta_1, \ldots, \ta_\sigma\}$ is either shifted by one to an even or odd position, resp. Let $\parity(\ta)\in \{\even, \odd\}$ denote the parity of the leftmost occurrence of the letter $\ta \in \Sigma$ in $\tilde{w}$. 
	
	\medskip
	
	Before continuing with proving that $w$ is strictly $k$-local iff $\tilde w = f_{|w|}(w)$ is strictly $k$-local,  we show the follwing intermediate result: Let $\ta\tb \in \Fact_2(w)$ and $\tb\ta\in\Fact_2(w)$ for $\ta,\tb \in \Sigma$ different (this assumption is in our setting w.l.o.g. since $w$ is condensed and the claim holds trivially for words of length $1$). If $\tilde{\ta} \in \al(\tilde{w})$ then for $\tilde{w}$ we either have that $\tilde{\ta}\ta\tb \in \Fact_3(w)$ or $\tb \tilde{\ta}\ta \in \Fact_3(w)$ holds $(*)$.
	For the sake of contradiction, suppose that both $\tilde{\ta}\ta\tb \in \Fact_3(w)$ and $\tb \tilde{\ta}\ta \in \Fact_3(w)$. W.l.o.g. assume that $\parity(\tb) = \even$. With $\tilde{\ta}\ta\tb \in \Fact_3(w)$ we conclude that $\parity(\ta) = \odd$ and by $\tb \tilde{\ta}\ta \in \Fact_3(w)$ we get $\parity(\ta) = \even$, a contradiction. 
	
	\medskip
	
	We now want to show that $w$ is strictly $k$-local iff $\tilde w = f_{|w|}(w)$ is strictly $k$-local. 
	For $\tilde w$ we construct a marking sequence as in Algorithm~\ref{alg:computing_tildesigma}. The main idea is to keep the order of the letters from $\{\ta_1, \ldots, \ta_\sigma\}$ in the marking sequence $\tilde{e}$ identical to $e$ and add shifting letters in between such that the number of blocks in $w$ and $\tilde{w}$ remains identical. Therefore, we always mark a letter $\ta \in \{\ta_1,\ldots, \ta_\sigma\}$ right before or after the marking of $\tilde{a}$. To determine the order of marking $\ta$ and $\tilde{\ta}$ we follow four cases:
	\begin{description}
		\item[Case 1:] If the letter $\ta$ is marked in the next step in $e$ and $\tilde{\ta}$ does not occur in $\tilde{w}$ we continue by marking $\ta$.
		\item[Case 2:] By $(*)$ we have that $\tilde{\ta}\ta\tb \in \Fact_3(\tilde{w})$ and $\tb\tilde{\ta}\ta \in \Fact_3(\tilde{w})$ exclude each other. If $\tb$ is marked in $e$ right before $\ta$ and $\tb\tilde{\ta}\ta \in \Fact_3(\tilde{w})$ holds then we continue $\tilde{e}$ by marking $\tilde{\ta}$ first and then $\ta$.
		\item[Case 3:] This case is dual to Case 2. If $\tb$ is marked in $e$ right before $\ta$ and $\tilde{\ta}\ta \tb \in \Fact_3(\tilde{w})$ holds then we continue $\tilde{e}$ by marking $\ta$ first and then $\tilde{\ta}$.
		\item[Case 4:] If none of the above cases applies, we can conclude that a letter in the marking sequence is chosen that is not in direct neighbourhood of a block. In this case, we obtain the same number of blocks by marking this letter in $\tilde{e}$, too and continue with its shifting letter (if this one exists).
	\end{description}
	\begin{algorithm}
		\caption{Compute $\tilde{e}$}\label{alg:computing_tildesigma}
		\Input{letter-square-free $w \in \Sigma^n$ and marking sequence $e$ which verifies that $w$ is strictly $k$-local for $k \in \N$}
		\Output{marking sequence $\tilde{e}$}
		\nonl\hrulefill
		
		$\tilde{w} :=  f_n(w)$\;
		$\tilde{e}[1] := e[1]$\;
		$j := 1$ \tcp*{counter $\text{for}$ the length of $\tilde{e}$}
		\tcp{mark $\tilde{a}$ $\text{and}$ $a$ always together}
		\If{$\tilde{a} \in \al(\tilde{w})$}{
			$j++$\;
			$\tilde{e}[j] := \tilde{a}$\;}
		\tcp{setting the next values of $\tilde{e}$ by iterating through $e$ and distinct on the relative positions of marked letters $a$ and their shifting letter $\tilde{a}$}
		\For{$i\in[2,|e|]$}
			{\tcp{Case 1}
			\If{$e[i] = a$ and $\tilde{a} \notin \al(\tilde{w})$}{
				$j++$\;
				$\tilde{e}[j] := a$\;}
			\tcp{Case 2}
			\ElseIf{
			$e[i] = a$ and $\tilde{a} \in \al(\tilde{w})$ and $\bar{\tilde{a}} \notin \al(\tilde{w}_j)$ and $e[i-1] \tilde{a} a \in \Fact_2(\tilde{w})$}{
				$j++$\;
				$\tilde{e}[j] := \tilde{a}$\;
				\tcp{mark $\tilde{a}$ $\text{and}$ $a$ always together}
				\If{$\bar{a} \notin \al(\tilde{w}_j)$}{
					$j++$\;
					$\tilde{e}[j] := a$\;}
			}
			\tcp{Case 3}
			\ElseIf{
			$e[i] = a$ and $\tilde{a} \in \al(\tilde{w})$ and $\bar{\tilde{a}} \notin \al(\tilde{w}_j)$ and $\tilde{a} a e[i-1] \in \Fact_2(\tilde{w})$}{
				$j++$\;
				$\tilde{e}[j] := a$\;
				\tcp{mark $\tilde{a}$ $\text{and}$ $a$ always together}
				\If{$\bar{a} \notin \al(\tilde{w}_j)$}{
					$j++$\;
					$\tilde{e}[j] := a$\;}
			}
			\tcp{Case 4}
			\Else{
				$j++$\;
				$\tilde{e}[j] := a$\;
				\tcp{mark $\tilde{a}$ $\text{and}$ $a$ always together}
				\If{$\bar{a} \notin \al(\tilde{w}_j)$}
					{$j++$\;
					$\tilde{e}[j] := a$\;}
			}
		}
		\Return $\tilde{e}$
\end{algorithm}
	Notice that Algorithm~\ref{alg:computing_tildesigma} terminates since the number of iterations is bounded by $|e|$. Moreover, the constructed marking sequence can be applied on $\tilde{w}$ since every time a letter is marked, its shifting letter is marked, too or v.v., i.e., $\tilde{e}$ does contain all the letters of $\tilde w$.
	
	For the correctness, we continue by showing that $w$ is strictly $k$-local iff $\tilde{w} := f_{|w|}(w)$ is strictly $k$-local.
	
	For the first direction let $w$ be strictly $k$-local. Thus, there exists a marking sequence $e$ that verifies the $k$-locality of $w$. We now want to show that $\tilde{w}$ is strictly $k$-local. First, we show $\tilde{w}$'s $k$-locality w.r.t. $\tilde{e}$ constructed via Algorithm~\ref{alg:computing_tildesigma} from $e$. In fact, we use induction on the marking steps in $e$ in order to show that $\tilde{w}$ does not contain more blocks than $w$ with the same amount of marked letters from $\{\ta_1, \ldots, \ta_\sigma\}$. As a second step, we continue with the strictness.
	
	For the base case, consider the first marking step $i= 1$. Since we have $\tilde{e}[1] = e[1]$ the number of blocks is identical for $w_1$ and $\tilde{w}[1]$. Furthermore, all occurrences of $\widetilde{e[1]}$ are placed left to $e[1]$'s in $\tilde{w}$ by construction. Thus, the number of blocks in $\tilde{w}$ does not increase in the second marking step with $\tilde{e}[1,2] = (e[1], \widetilde{e[1]})$ if $\widetilde{e[1]}\in \al(\tilde{w})$.
	
	For the inductive step, we want to show that the \nth{$i$} marking step of $e$ in $w$ does not result in more than $k$ blocks in $\tilde{w}$ w.r.t. $\tilde{e}$ under the assumption that for all marking steps $i' < i$ of $e$ in $w$, $\tilde{w}$ has at most $k$ blocks w.r.t. $\tilde{e}$.
	Let $i \in [|e|]$ be an arbitrary marking step. By the definition of $k$-locality $w_j$ does not contain more than $k$ blocks. We now distinct on the relative positions of $e[i]$ and marked letters.
	\begin{description}
		\item[Case 1:] If at least one occurrence of $e[i]$ in $w$ is neighboured to an already marked letter in $w_{i-1}$, then we can conclude that either Case 1, 2 or 3 of Algorithm~\ref{alg:computing_tildesigma} apply. In Case 1, $\tilde{w}$ does not contain the shifting letter for $e[i]$, i.e., the number of blocks in $w$ and $\tilde{w}$ is identical using the inductive hypothesis. In both, Case 2 and 3, $e[i]$ and $\widetilde{e[i]}$ are marked in the resp. order. Since $(*)$ shows that those cases are mutually exclusive and every occurrence of $e[i]$ in $w$ that does not occur next to a marked letter opens a new block in $w$ and by the inductive hypothesis, it follows that $\tilde{w}$ has the same number of blocks after marking $\al(e[1,i])$ and $\{\widetilde{e[1]}, \ldots, \widetilde{e[i]}\}$.
		\item[Case 2:] If none of the occurrences of $e[i]$ in $w$ is neighboured to a marked letter in $w_{i-1}$ then either Case 1 or 4 of Algorithm~\ref{alg:computing_tildesigma} apply. In Case 1, $\tilde{w}$ does not contain the shifting letter for $e[i]$, i.e., the number of blocks in $w$ and $\tilde{w}$ is identical. In Case 4, $e[i]$ is a letter that occurs in $\tilde{w}$ also in its shifted form and is not neighboured by a marked block. Thus, the marking of $e[i]$ in $\tilde{w}$ results in the same number of blocks for $w$ and $\tilde{w}$ using the inductive hypothesis. By the same argument as given in the inductive base, the number of blocks does not increase when marking $\widetilde{e[i]}$ in the following step.
	\end{description}
	In total, this shows that $\tilde{w}$ contains at most $k$ blocks when being marked with $\tilde{e}$ from Algorithm~\ref{alg:computing_tildesigma}, i.e., the $k$-locality of $\tilde{w}$.
	
	Now, we proceed by showing that $\tilde{w}$ is strictly $k$-local. For the sake of contradiciton, suppose that $\tilde{w}$ is not strictly $k$-local. Since we have shown its $k$-locality in the previous step this means that there exists a $k' <k$ such that $\tilde{w}$ is $k'$-local and a respective marking sequence $\tilde{w}'$. By only considering letters from $\{\ta_1,\ldots, \ta_\sigma\}$ we obtain the marking sequence $\pi_{\{\ta_1, \ldots, \ta_\sigma\}}(\tilde{e}')$ that marks $w$ with at most $k'$ blocks per marking step, a contradiction to $w$ being strictly $k$-local. 
	
	\medskip
	
	For the second direction, let $\tilde{w}$ be strictly $k$-local. Suppose that $w$ is not strictly $k$-local. In the case that $w$ is $k''$-local for a $k'' < k$ we can conclude that $\tilde{w}$ is $k''$-local two using the first direction. This contradicts the fact that $\tilde{w}$ is $k$-local. Thus, there exists a $k' > k$ such that $w$ is $k'$-local. Let $\tilde{e}$ be any marking sequence verifying that $\tilde{w}$ is $k$-local. Using their projection $\pi_{\{\ta_1, \ldots, \ta_\sigma\}}$, i.e., deleting all letters but $\{\ta_1,\ldots,\ta_{\sigma}\}$, for a marking of $w$ we can conclude that $w$ is $k$-local since all letters from $\{{\ta_1}, \ldots, {\ta_\sigma}\}$ are located next to their shifting letter, a contradiction. 
	
	In total, we have shown the existence of a word with the same locality and the property that every letter does purely occur on even or odd positions for every letter-square-free word.
\end{proof}

\fi

\begin{theorem}\label{biparbitrary}
	Let $G$ be a bipartite graph that is coloured by $c$ with more than $\ell>2$ colours. Then $(G,c,k)\in\Loc_G$ is $\NP$-complete.
\end{theorem}
\begin{proof}
	First of all, the $\NP$-membership follows directly from the NP-membership for arbitrary graphs. To show the problems $\NP$-hardness, we reduce from the problem $\Loc_w$ that was used for Proposition~\ref{completeness} before. Let $w' \in \Sigma^*$ and $k \in \N$ such that $(w',k) \in \Loc_w$. By  Proposition~\ref{tildewconstruction} and its proof we can conclude that for $w = \cond(w')$ and $\tilde{w} = f_{|w|}(w)$ with $f$ from the proof of Proposition~\ref{tildewconstruction} we have 
	\[
	(w',k) \in \Loc_w \iff (w,k) \in \Loc_w \iff (\tilde{w},k) \in \Loc_w.
	\]
	Since all modifications are polynomial-time computable and $\tilde{w}$ a linear representation a bipartite graph coloured with $\al(\tilde{w})$ colours the $\NP$-completeness follows by Proposition~\ref{completeness}.
\end{proof}

Now, we investigate snarks, i.e., connected bridgeless cubic graphs with a minimal number of colours  for colouring edges of $4$ (cf. Gardner \cite{snarks}). In 1880, Tait showed that the famous four-colour theorem holds true iff every snark is non-planar \cite{tait}. One of the most famous snarks is the Peterson graph.

\begin{remark}
	Note that the Peterson graph is $3$-local since it is $3$-colourable and there exists a colour $x$ that occurs exactly $3$ times for every $3$-colouring.
\end{remark}

\ifpaper 
	\begin{figure}
	\centering
	\begin{tikzpicture}[scale=0.9]
		\node[shape=circle,draw=black,fill=cyan] (u11) at (0,0) {\tiny$u_1^1$};
		\node[shape=circle,draw=black,fill=red] (u21) at (0.7,-0.7) {\tiny$u_2^1$};
		\node[shape=circle,draw=black,fill=yellow] (u31) at (0.5,-1.6) {\tiny$u_3^1$};
		\node[shape=circle,draw=black,fill=cyan] (u41) at (-0.5,-1.6) {\tiny$u_4^1$};
		\node[shape=circle,draw=black,fill=red] (u51) at (-0.7,-0.7) {\tiny$u_5^1$};
		
		\node[shape=circle,draw=black,fill=red] (v1) at (0,1) {\tiny$v_1$} ;
		\node[shape=circle,draw=black,fill=yellow] (v2) at (1.7,-0.4) {\tiny$v_2$} ;
		\node[shape=circle,draw=black,fill=red] (v3) at (1,-2.4) {\tiny$v_3$} ;
		\node[shape=circle,draw=black,fill=red] (v4) at (-1,-2.4) {\tiny$v_4$} ;
		\node[shape=circle,draw=black,fill=yellow] (v5) at (-1.6,-0.4) {\tiny$v_5$} ;
		
		\node[shape=circle,draw=black,fill=cyan] (u12) at (-0.6,1.6) {\tiny$u_1^2$};
		\node[shape=circle,draw=black,fill=cyan] (u13) at (0.6,1.6) {\tiny$u_1^3$};
		\node[shape=circle,draw=black,fill=red] (u22) at (2.3,0.2) {\tiny$u_2^2$};
		\node[shape=circle,draw=black,fill=red] (u23) at (2.5,-0.7) {\tiny$u_2^3$};
		\node[shape=circle,draw=black,fill=yellow] (u32) at (1.8,-2.8) {\tiny$u_3^2$};
		\node[shape=circle,draw=black,fill=yellow] (u33) at (0.9,-3.2) {\tiny$u_3^3$};
		\node[shape=circle,draw=black,fill=cyan] (u42) at (-1.1,-3.2) {\tiny$u_4^2$};
		\node[shape=circle,draw=black,fill=cyan] (u43) at (-1.8,-2.5) {\tiny$u_4^3$};
		\node[shape=circle,draw=black,fill=red] (u52) at (-2.5,-0.7) {\tiny$u_5^2$};
		\node[shape=circle,draw=black,fill=red] (u53) at (-2.2,0.2) {\tiny$u_5^3$};
		
		\path [-] (u11) edge node[left] {} (u21);
		\path [-](u21) edge node[left] {} (u31);
		\path [-](u31) edge node[left] {} (u41);
		\path [-](u41) edge node[left] {} (u51);
		\path [-](u51) edge node[left] {} (u11);
		\path [-](u11) edge node[left] {} (v1);
		\path [-](u21) edge node[left] {} (v2);
		\path [-](u31) edge node[left] {} (v3);
		\path [-](u41) edge node[right] {} (v4);
		\path [-](u51) edge node[right] {} (v5);   
		\path [-](v1) edge node[right] {} (u12);  
		\path [-](v1) edge node[right] {} (u13);   
		\path [-](v2) edge node[right] {} (u22);  
		\path [-](v2) edge node[right] {} (u23);   
		\path [-](v3) edge node[right] {} (u32);  
		\path [-](v3) edge node[right] {} (u33);   
		\path [-](v4) edge node[right] {} (u42);  
		\path [-](v4) edge node[right] {} (u43);   
		\path [-](v5) edge node[right] {} (u52);  
		\path [-](v5) edge node[right] {} (u53);   
		
		\path [-](u12) edge [bend left=80] node[] {} (u22);
		\path [-](u22) edge [bend left=80] node {} (u32);
		\path [-](u32) edge [bend left=80] node[right] {} (u42);
		\path [-](u42) edge [bend left=80] node[right] {} (u52);
		\path [-](u52) edge [bend left=80] node[right] {} (u13);
		\path [-](u13) edge [bend left=80] node[right] {} (u23);
		\path [-](u23) edge [bend left=80] node[right] {} (u33);
		\path [-](u33) edge [bend left=80] node[right] {} (u43);
		\path [-](u43) edge [bend left=80] node[right] {} (u53);
		\path [-](u12) edge node[right] {} (u53);   
	\end{tikzpicture}
	\caption{Flower snark $J_5$ with a $3$-colouring.}
	\label{flowersnarks}
\end{figure}
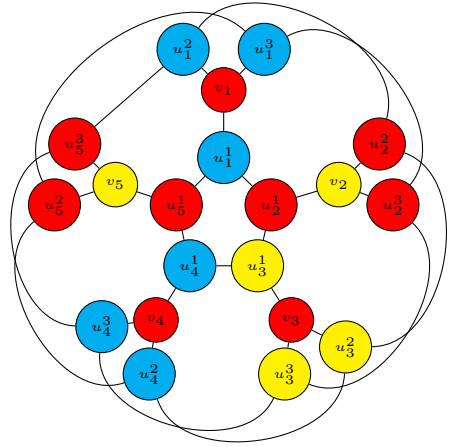

\else    

\fi

\begin{remark}
	Note that the chromatic number of flower snarks is $3$. See Figure~\ref{flowersnarks} for en exemplary $3$-colouring of $J_5$.
\end{remark}

\begin{proposition}
	For every flower snark $J_d$, $d\in 2\N +1$ there exists a $\chi(J_d)$-colouring such that $J_d$ is $(d - (d \operatorname{mod} 3))$-local.
\end{proposition}
\ifpaper 
\else    
\begin{proof}
	First, we give a $(\chi(J_d) = 3)$-colouring for $J_d$. Define the colouring function $c: V \mapsto \{r,g,b\}$ for $r,g,b \in \N$ distinct as follows:
	\begin{align*}
		& v \mapsto c(v) \\
		&= \begin{cases}
			b, & \text{ if } v \in \{u_i^1,u_i^2,u_i^3 \mid i \in \{3d+1 \mid d \in [(d \operatorname{div} 3)-1]_0\}\},\\
			r, & \text{ if } v \in \{v_i \mid i \in \{3d+1 \mid d \in [(d \operatorname{div} 3)-1]_0\}\},\\
			r, & \text{ if } v \in \{u_i^1,u_i^2,u_i^3 \mid i \in \{3d+2 \mid d \in [(d \operatorname{div} 3)-1]_0\}\},\\
			g, & \text{ if } v \in \{v_i \mid i \in \{3d+2 \mid d \in [(d \operatorname{div} 3)-1]_0\}\},\\
			g, & \text{ if } v \in \{u_i^1,u_i^2,u_i^3 \mid i \in \{3d+3 \mid d \in [(d \operatorname{div} 3)-1]_0\}\},\\
			r, & \text{ if } v \in \{v_i \mid i \in \{3d+3 \mid d \in [(d \operatorname{div} 3)-1]_0\}\}.
		\end{cases}
	\end{align*}
	First, we observe that all $v_i$ are coloured differently than the adjacent vertices $u_i^1,u_i^2$ and $u_i^3$ for all $i \in [d]$. Further, the inner cycle $u_1^1\cdots u_d^1$ is coloured alternately in $b,r$ and $g$. The two outer vertices of every star graph are coloured identically as the inner cycle, resulting in an alternating colouring between $r$ and $g$.
	So, by the definition of $J_d$ this is a valid $3$-colouring. 
	
	Now, we need to find a marking sequence that verifies $J_d$'s $(d - (d \operatorname{mod} 3))$-locality. Consider $e = (b,r,g)$. In the first marking step, we have $(d - (d \operatorname{mod} 3))$ marked vertices since for every element in the set $[(d \operatorname{div} 3) -1]_0$ we have three vertices that are $b$-coloured. By now marking $r$ in the second marking step, we distinct three different cases:
	\begin{description}
		\item[Case 1 ($d \operatorname{mod} 3 =0$):] In this case, the inner circle can alternate every colour equally, so every set of vertices $\{u_i^1, v_i, u_i^2, u_i^3, u_{i+1}^1,  u_{i+1}^2, u_{i+1}^3\}$ for $i \in \{3d+1 \mid d \in [(d \operatorname{div} 3)-1]_0\}$ is a connected component. Every of them is separated via a $1$-vertex connected component $\{v_i\}$ for every $i \in \{3d+3 \mid d \in [(d \operatorname{div} 3)-1]_0\}\}$. Those are in total $2(d \operatorname{div} 3)$ connected components and by $(d - (d \operatorname{mod} 3)) > 2(d \operatorname{div} 3)$, the $(d - (d \operatorname{mod} 3))$-locality of $J_d$ follows.
		\item[Case 2 ($d \operatorname{mod} 3 =1$):] In this case the outer vertices of the star graph with central vertex $v_1$ are adjacent to the vertices $u_1^2, u_1^3,u_d^2, u_d^3$ that are all coloured in $r$ by $c$. Therefore, the marking of $b$ and $r$ results in the connected two components $\{u_1^1, u_2^1, u_d^1, v_1,u_1^2, u_1^3,u_2^2, u_2^3, u_d^2, u_d^3\}$ and $\{u_i^1, v_i, u_i^2, u_i^3, u_{i+1}^1,  u_{i+1}^2, u_{i+1}^3\}$ for $i \in \{3d+1 \mid d \in [(d \operatorname{div} 3)-1]\}$. Again, between every two of those connected components there lays an $r$-coloured vertex by construction which forms a connected component on its own. Thus, we result in $2(d \operatorname{div} 3)$ connected components in total. The result follows as in Case 1.
		\item[Case 3 ($d \operatorname{mod} 3 =2$):] In this case the marking of $b$ and $r$ results in the connected components $\{u_1^1, u_2^1, u_d^1, v_1,u_1^2, u_1^3,u_2^2, u_2^3, u_d^2, u_d^3, u_{d-1}^1, u_{d-1}^2,u_{d-1}^3\}$ and $\{u_i^1, v_i, u_i^2, u_i^3, u_{i+1}^1,  u_{i+1}^2, u_{i+1}^3\}$ for $i \in \{3d+1 \mid d \in [(d \operatorname{div} 3)-1]\}$. Again, between every two of those connected components there lays an $r$-coloured vertex by construction which forms a connected component on its own. Thus, we result in $2(d \operatorname{div} 3)$ connected components in total. The result follows as in Case 1.
	\end{description}
	See Figure~\ref{flowersnarks} for an exemplary colouring given with $c$ for $J_5$ and the connected components as in Case 3.
	The $(d - (d \operatorname{mod} 3))$-locality of $J_d$ follows.
\end{proof}

\fi

\begin{remark}
	Note that there also exists a colouring of $J_n$ that uses the most infrequent colour (here $b$) only three times. This colouring is given by
	\begin{align*}
		c(v) 
		= \begin{cases}
			b, & \text{ if } v \in \{u_1^1,u_1^2,u_1^3\},\\
			r, & \text{ if } v \in \{u_i^1, u_i^2, u_i^3 \mid i \in [n-1] \land i \text{ odd}\} \cup \{v_i \mid i \text{ even}\},\\
			b, & \text{ if } v \in \{u_i^1, u_i^2, u_i^3 \mid i \in [n] \land i \text{ odd}\} \cup \{v_1\} \cup \{v_i \mid i \text{ even}\}.
		\end{cases}
	\end{align*}
	Nevertheless, this colouring results in a higher locality since not all of the $r$- or $g$-coloured vertices are adjacent to the three $b$-coloured vertices. 
\end{remark}

\begin{remark}
The reader may verify that for the first and second Blanuša snark and the double-star snark there exist $3$-colourings such that they are $5$-, $6$- and $9$-local, resp. 
\end{remark}

\begin{table}
	\centering
	\caption{A summary of upper and lower bounds for the locality of graph classes.}
	\vspace{3mm}
	\begin{tabular}{l|c|c}
		Class & Lower Bound & Upper Bound \\
		\hline
		\hline
		Complete Graphs, Star Graphs& $1$ & $1$\\
		Wheel Graphs, Friendship Graphs & &  \\
		\hline
		Paths & $1$ & $\lfloor n/2 \rfloor$ \\
		Cycles, Web graphs &  & \\
		1/2-regular Graphs & & \\
		\hline
		Complete Bipartite Graphs $K_{n_1,n_2}$ & $1$ & $\min\{n_1,n_2\}$\\
		Crown Graphs, Hypercubes & & \\
		Bipartite Graphs with $2$-colouring & & \\
		Knight's graph, Gear graph & & \\
		\hline
		Sunflower graphs $S_d$, Helm graphs $H_d$ & 1 & $\lfloor \frac{d-1}{2} \rfloor$ \\
		\hline
		Peterson Graph & 1 & 3 \\
		\hline
		Edgeless Graphs, 0-regular Graphs & $n$ & $n$\\
	\end{tabular}
	\label{tablegraphs}
\end{table}

	\section{Practical Application}\label{pract}

In addition to theoretical properties, we are also interested in practical applications of $k$-locality in graphs.
One application could be to optimize utility of a location of a shop type (groceries, pharmacy, restaurant, etc.) in a road network by determining the largest number of involved features (types of buildings) in a $k$ block radius.
This could make customer trips more convenient, since they could visit more points of interest in a shorter trip.
Further, $k$-locality could be used as part of data exploration, by yielding valuable insights into the organising principles of the underlying graph.

In order to calculate the minimal $k$-locality of a given graph, the na\"ive approach would be to enumerate all valid colour permutations, and to compute the $k$-locality w.r.t. every marking sequence.
However, since this approach would not scale well for increasing numbers of colours or larger graphs, we propose a priority-based search (cf. Algorithm \ref{alg:priority_search}) to make the calculation feasible.
First, we initialise the priority queue with an empty marking prefix $\varepsilon$ and max$_k = 0$ (line 3).
The queue is maintained in ascending order of the $k$-locality of the respective marking prefixes.
Since the first queue element marks the prefix with the least number of components, we pop it as the most promising candidate.
For each remaining colour we expand the current sequence prefix by one colour, count the resulting components, and sort the expanded sequences back into the queue (lines 6-9).
Once a sequence contains all colours, it is complete and the minimal $k$-locality may be updated (lines 11-15).
Note that this algorithm finds \textit{all} sequences that have minimal $k$-locality.
If one is only interested in finding the minimal $k$-locality, it is sufficient to find only one of the sequences with minimal $k$-locality.
This can be achieved by requiring the first queue element to have strictly fewer connected components than the current best ('$<$' instead of '$\leq$' in line 4), and only keeping a single best sequence (removing lines 11 and 12).
We demonstrate experimentally that this yields even further improvements in runtime.

\begin{theorem}\label{priority_search_optimality}
    Algorithm \ref{alg:priority_search} is optimal in the number of successive expansions of a marking prefix beginning with an empty sequence and is guaranteed to find the correct solution.
\end{theorem}

\begin{algorithm}[!h]
\scriptsize
\caption{Priority Search}
\label{alg:priority_search}
    \Input{Graph $G = (V, E)$, valid colouring $c : V \to [\ell]$}
    \Output{minimum $k$-locality of $G$, all marking sequences of that $k$-locality}
    \nonl\hrulefill
    
    min$_k \gets \infty$ \;
    minSequences $\gets \emptyset$ \;
    queue $\gets \left[ (0, \varepsilon) \right]$    \tcp*{\scriptsize $\left[ \, (\text{max}_k, \, \text{prefix}), \, \dots \right]$ sorted asc. by $\text{max}_k$}
    
    \While{queue $\neq \emptyset \, \land$ queue$\left[0\right]\left[0\right] \leq$ min$_k$}{
        (current$_k$, prefix) $\gets$ queue.pop() \;
        \If{prefix incomplete}{
                        \ForAll{remaining colours $c_j$ not in prefix}{
                count $\gets$ no. components in $G$ marked with prefix $+ \left[ c_j \right]$\;
                queue.insort($(\text{max}(\text{current}_k, \text{count}), \text{prefix} + \left[ c_j \right])$)
            }
        } \Else (\tcp*[f]{\scriptsize current sequence is complete}){
            \If{current$_k$ = min$_k$}{
                minSequences $\gets$ minSequences $\cup \, \{ \text{prefix} \}$ \;
            }
            \If{current$_k$ < min$_k$}{
                min$_k \gets$ current$_k$ \;
                minSequences $\gets \{ \text{prefix} \}$ \;
            }
        }
    }
    \Return min$_k$, minSequences \;
\end{algorithm}

In the following we further empirically demonstrate the efficiency of the algorithm for various graphs.
Many systems in nature and society, such as biological neural networks, bibliographical networks, or the internet, frequently exhibit a common organising principle (next to sheer size), which is referred to scale-free power-law distributed vertex connectivity \cite{barabasi1999emergence,albert2002complexNetworks}.
Barabási and Albert identified two main reasons that result in this phenomenon: (1) new vertices enter the network, causing it to expand over time, and (2) new vertices tend to connect to already well-connected vertices (a.k.a. preferential attachment or rich-get-richer) \cite{barabasi1999emergence,albert2002complexNetworks}.
Based on these principles they developed a random graph model, which creates scale-free power-law distributed graphs \cite{barabasi1999emergence,albert2002complexNetworks}.
Since the Barabási-Albert random graph represents a large class of real-world graphs, we use it as a proxy to benchmark the efficiency and scalability of our approach on graphs with good control over the number of vertices and colours, while still retaining typical, real-world properties (such as degree distribution, diameter, and clustering coefficient).
Due to space considerations, we neglect other random graph models, such as the Erd\H{o}s-Rényi model \cite{erdos1960evolution} and the Watts-Strogatz model \cite{watts1998collective}, since these generally feature less natural degree distributions, although they would lend themselves for considerations in future work for more extensive comparisons.
We vary the network size from 10k to 100k vertices, where each vertex added connects to 10 existing vertices.
The vertex colours are distributed uniformly at random.
We investigate the impact of graph size and number of colours separately to avoid excessive parameter combinations that yield little additional insight.
Therefore, when measuring the impact of the graph size alone, we keep the number of colours fixed at 5, and for impact of the number of colours fix the number of vertices at 5000.
Different combinations of these parameters yield a similar result, which we omit for the sake of brevity.
We report the average runtime with standard deviation over 5 runs of the na\"ive approach, the priority search (Algorithm \ref{alg:priority_search}), as well as a variant of Algorithm \ref{alg:priority_search}, named priority*, which does not yield \textit{all} marking sequences with minimal $k$-locality, but terminates once it has found at least one that is guaranteed to have minimal $k$-locality.
Figure~\ref{fig:sensitivity_n} shows that the priority-based search outperforms the na\"ive approach substantially and scales better to larger graphs.
Since counting the number of components takes longer for larger graphs, the priority-based searches benefit from requiring fewer counts.
This is drastically emphasised in Figure~\ref{fig:sensitivity_colour}, which shows that the na\"ive search does not scale well for increasing numbers of colours (note the log-scale).
Depending on the application scenario it may be preferable to chose priority* over the full variant, since it can shave off even more computation time in case only the minimum $k$-locality is needed, and outperforms the na\"ive approach by up to 4 orders of magnitude.
However, there may still be a benefit particularly for exploration purposes to have the ability to find all sequences with minimal $k$-locality.

\begin{figure}[!h]
    \centering
    \begin{subfigure}[b]{0.49\textwidth}
        \centering
        \includegraphics[width=\linewidth]{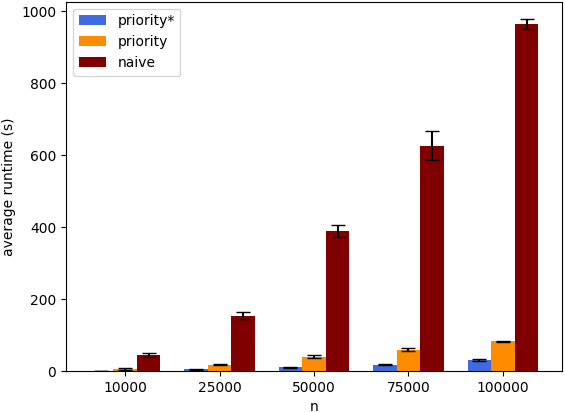}
        \caption{Sensitivity to $|V|$}
        \label{fig:sensitivity_n}
    \end{subfigure}
    \hfill
    \begin{subfigure}[b]{0.49\textwidth}
        \centering
        \includegraphics[width=\linewidth]{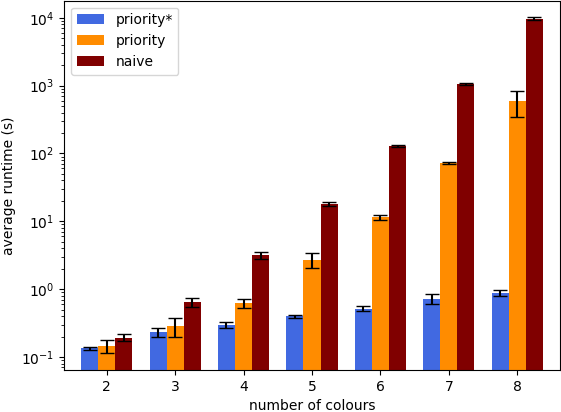}
        \caption{Sensitivity to number of colours.}
        \label{fig:sensitivity_colour}
    \end{subfigure}
    \caption{Efficiency and scalability on Barabási-Albert random graphs}
    \label{fig:efficiency} 
\end{figure}

One limitation of the previous experiments is the uniform distribution of colours in the graph.
In real-world, \textit{heterogeneous} graphs \cite{sun2011pathsim,HIN-survey}, where the colours could represent vertex types, the distribution would not be random, but rather organised by some underlying principle.
As a familiar example, we look at a commonly studied subset of the DBLP publication graph, that covers the 4 areas of \textit{databases}, \textit{data mining}, \textit{machine learning}, and \textit{information retrieval} \cite{sun2011pathsim,HIN-survey}.
The graph contains about 5.9k author vertices ($A$), 5.2k papers ($P$), 18 venues $V$, and 4.4k topics ($T$), which are linked following the network schema in Figure~\ref{fig:dblp} (left).
The average runtime of the search algorithms over 5 runs is shown in Figure~\ref{fig:dblp} (right), where the speed-up of the priority-based search again is substantial (almost 2 orders of magnitude), despite only considering 4 colours/types.
Since only two marking sequences with minimal $k$-locality exist, the difference between priority and priority* is negligible.

In addition to the efficiency, the computed result itself is of great interest.
The minimum $k$-locality of the DBLP subset is 18, with the two minimal sequences being $(V, P, A, T)$ and $(V, P, T, A)$.
This suggests that the 18 venues, despite being few in numbers, play a pivotal role in the network and act as hubs, which span large areas of the network through the connecting papers.
This matches the expectation one would have for such a graph, since the publishing activity revolves around the venues.
The next smallest sequences have $k$-localities larger than 4.4k, starting with $T$.
From this we can infer that no other vertex types act in a similarly central capacity as the venues, but instead are scattered across the graph.
This hierarchy of centrality is not dissimilar to the underlying principles of hierarchical clustering.
We believe that further study of the $k$-locality and the corresponding colour/type sequences can provide great insights for data exploration.
A similar evaluation of other areas of computer science like theoretical computer science is ajar but not of less interest.
The choice for our sample subset of the DBLP graph is based on the experience of some authors with the evaluation of this particular subgraph, and demonstrates the capability of $k$-locality to be a helpful tool in data exploration.

\begin{figure}[!h]
    \centering
    \begin{subfigure}[b]{0.49\textwidth}
        \centering
        \raisebox{.75cm}{
            \includegraphics[width=0.8\linewidth]{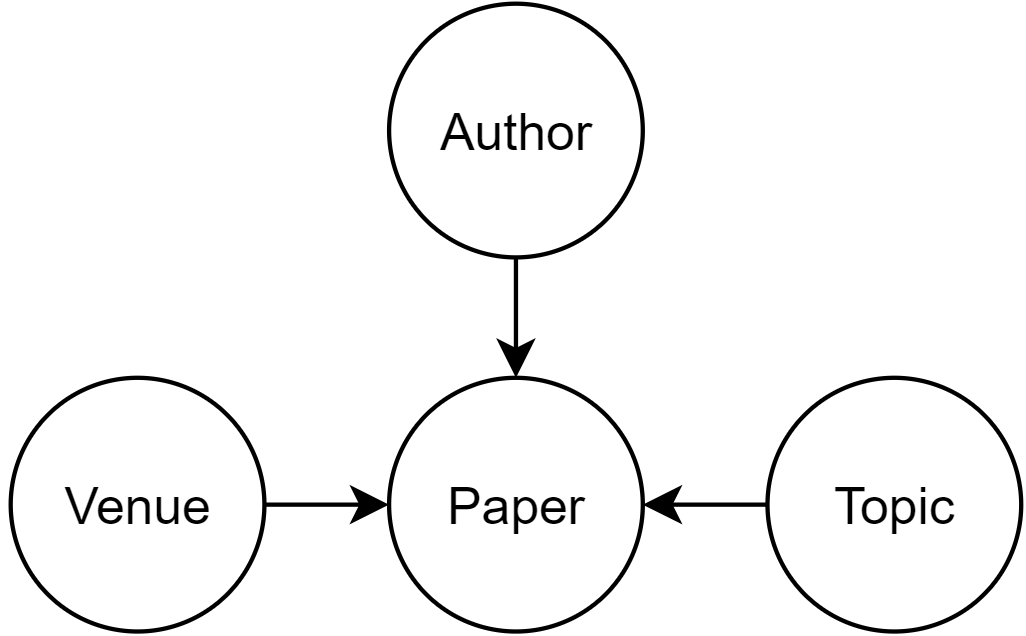}
        }
    \end{subfigure}
    \hfill
    \begin{subfigure}[b]{0.49\textwidth}
        \centering
        \includegraphics[width=\linewidth]{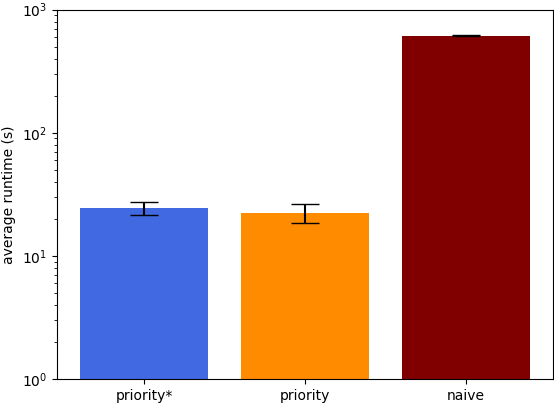}
    \end{subfigure}
    \caption{Schema of the DBLP graph (left) and runtime comparison (right).}
    \label{fig:dblp} 
\end{figure}

	\section{Conclusions}\label{conc}
	
In this work, we extended the notion of $k$-local words introduced by Day et al.~\cite{localpatterns} to graphs.
Not astonishingly the general problem of deciding whether a graph is $k$-local for some $k\in\N$ is $\NP$-complete since words can be interpreted as linear graphs.
Nevertheless, some graph classes yield a determinable locality based on a minimal colouring.
In this field, it remains open to determine the exact colouring since not all minimal colourings lead to the same locality.
As mentioned in Remark~\ref{everygraph1local} every connected graph is $1$-local colourable if the number of colours is not limited.
Fixing $\ell$ and characterising the graphs for which an $\ell$-local colouring exists, is a topic for further research. For several graph classes, we provided the minimal and maximal locality based on a minimal colouring of the graphs.
The theoretical results so far follow the notion of $k$-locality on words, i.e., find an enumeration such that the $k$-locality is minimal.
Regarding pattern mining, one could also alter the question in finding the maximum:
Find a route, more explicitly a sequence of supermarkets and pharmacies to be visited, which contains as many co-located supermarkets and pharmacies together in a specific range, i.e.,  maximise $k$ that denotes the number of co-occurring {\em blocks} of supermarkets and pharmacies.
The benefit that comes with this task, is that users have multiple options if one co-location of supermarkets and pharmacies is too crowded (long queues, many items out of stock, etc.) the closest options are provided.
Another application is that for a given $k$ determine the largest number of involved features $m$ (types of buildings) involved, i.e., get me for $k=2$ blocks the largest number of different shop types (groceries, pharmacy, fast-food restaurant etc.).
The purpose behind this application is to maximise the \textit{utility} of a specific location by having as many different shop types $d$ as the customer needs while at the same time having $k-1$ further locations to visit where the $d$ shop types co-occur.
This allows for minimising travel time by having multiple shop types at the customers' disposal per location.
In order to enable such tasks for the future, we have proposed a priority search algorithm, which is prefix expansion-optimal and able to calculate the minimum $k$-locality of a graph and its corresponding marking sequences several orders of magnitude faster than an exhaustive search on scale-free graphs.
Since scale-free graphs are ubiquitous \cite{barabasi1999emergence,albert2002complexNetworks}, we believe our solution bears great potential for future studies, not only in the theoretical field, but also in the practical.

\subsubsection*{Acknowledgements.} We would like to thank Max Bannach, Malte Skambath and Philipp Sieweck for helpful discussions and some main ideas included in this work.

	%
	%
	%
	 \bibliographystyle{splncs04}
	 \bibliography{refs}

\begin{thebibliography}{10}
\providecommand{\url}[1]{\texttt{#1}}
\providecommand{\urlprefix}{URL }
\providecommand{\doi}[1]{https://doi.org/#1}

\bibitem{albert2002complexNetworks}
Albert, R., Barab\'asi, A.L.: Statistical mechanics of complex networks. Rev.
  Mod. Phys.  \textbf{74},  47--97 (2002)

\bibitem{DBLP:journals/jcss/Angluin80}
Angluin, D.: Finding patterns common to a set of strings. J. Comput. Syst. Sci.
   \textbf{21}(1),  46--62 (1980)

\bibitem{barabasi1999emergence}
Barab{\'a}si, A.L., Albert, R.: Emergence of scaling in random networks.
  Science  \textbf{286}(5439),  509--512 (1999)

\bibitem{cutwidth}
Casel, K., Day, J.D., Fleischmann, P., Kociumaka, T., Manea, F., Schmid, M.L.:
  Graph and string parameters: Connections between pathwidth, cutwidth and the
  locality number. In: {ICALP}. LIPIcs, vol.~132, pp. 109:1--109:16 (2019)

\bibitem{flowersnarks}
Clark, L., Entringer, R.: Smallest maximally nonhamiltonian graphs. Periodica
  Mathematica Hungarica  \textbf{14}(1),  57--68 (1983)

\bibitem{crochemore}
Crochemore, M., Hancart, C., Lecroq, T.: Algorithms on strings. Cambridge
  University Press (2007)

\bibitem{localpatterns}
Day, J.D., Fleischmann, P., Manea, F., Nowotka, D.: Local patterns. In:
  {IARCS}. LIPIcs, vol.~93, pp. 24:1--24:14 (2017)

\bibitem{erdos1960evolution}
Erd\H{o}s, P., R{\'e}nyi, A., et~al.: On the evolution of random graphs. Publ.
  math. inst. hung. acad. sci  \textbf{5}(1),  17--60 (1960)

\bibitem{SofsemTim}
Fleischmann, P., Haschke, L., L{\"{o}}ck, T., Nowotka, D.: Word-representable
  graphs from a word's perspective. In: {SOFSEM} 2024. LNCS, vol. 14519, pp.
  255--268. Springer (2024 (accepted at Acta Informatica))

\bibitem{fleischmann}
Fleischmann, P., Haschke, L., Manea, F., Nowotka, D., Tsida, C.T., Wiedenbeck,
  J.: Blocksequences of k-local words. In: SOFSEM 2021. Springer International
  Publishing (2021)

\bibitem{snarks}
Gardner, M.: Mathematical games. Scientific American  \textbf{234}(4),
  126--131 (1976)

\bibitem{hamady2008motifcluster}
Hamady, M., Widmann, J., Copley, S., Knight, R.: Motifcluster: an interactive
  online tool for clustering and visualizing sequences using shared motifs.
  Genome biology  \textbf{9},  1--13 (2008)

\bibitem{Isaacs1975InfiniteFO}
Isaacs, R.: Infinite families of nontrivial trivalent graphs which are not tait
  colorable. Am. Math. Mon.  \textbf{82},  221--239 (1975)

\bibitem{Mamoulis2008}
Mamoulis, N.: Co-location Patterns, Algorithms, pp. 103--107. Springer US,
  Boston, MA (2008)

\bibitem{masrur2019co}
Masrur, A., Thakur, G., Sparks, K., P., R., Peuquet, D.: Co-location pattern
  mining of geosocial data to characterize urban functional spaces. In: Big
  Data. pp. 4099--4102. IEEE (2019)

\bibitem{morimoto2001mining}
Morimoto, Y.: Mining frequent neighboring class sets in spatial databases. In:
  SIGKDD. pp. 353--358 (2001)

\bibitem{rousseeuw1987silhouettes}
Rousseeuw, P.: Silhouettes: a graphical aid to the interpretation and
  validation of cluster analysis. J. Comput. Appl. Math.  \textbf{20},  53--65
  (1987)

\bibitem{shekhar2001discovering}
Shekhar, S., Huang, Y.: Discovering spatial co-location patterns: A summary of
  results. In: SSTD. pp. 236--256. Springer (2001)

\bibitem{HIN-survey}
Shi, C., Li, Y., Zhang, J., Sun, Y., Philip, S.Y.: A survey of heterogeneous
  information network analysis. IEEE Transactions on Knowledge and Data
  Engineering  \textbf{29}(1),  17--37 (2016)

\bibitem{shinohara1982polynomial}
Shinohara, T.: Polynomial time inference of pattern languages and its
  applications. In: MFCS. pp. 191--209 (1982)

\bibitem{sibson1973slink}
Sibson, R.: Slink: an optimally efficient algorithm for the single-link cluster
  method. Comput. J.  \textbf{16}(1),  30--34 (1973)

\bibitem{sun2011pathsim}
Sun, Y., Han, J., Yan, X., Yu, P.S., Wu, T.: Pathsim: Meta path-based top-k
  similarity search in heterogeneous information networks. Proceedings of the
  VLDB Endowment  \textbf{4}(11),  992--1003 (2011)

\bibitem{tait}
Tait, P.: Remarks on the colourings of maps. Proc. R. Soc. Edinburgh
  \textbf{10},  729–729 (1880)

\bibitem{watts1998collective}
Watts, D.J., Strogatz, S.H.: Collective dynamics of ‘small-world’networks.
  nature  \textbf{393}(6684),  440--442 (1998)

\bibitem{yu2022method}
Yu, Q., Zhang, X., Hu, Y., Chen, S., Yang, L.: A method for predicting dna
  motif length based on deep learning. EEE/ACM Trans. Comput. Biol. Bioinform.
  \textbf{20}(1),  61--73 (2022)

\end{thebibliography}

%
 \section{Further Definitions}\label{furtherdefs}
 \subsection{$k$-Locality on Words}

Given a property $P:\Sigma\rightarrow\{0,1\}$, a factor $u$ is a {\em $P$-block}  of a word $w=xuy$ if $P(u[i])=1$ for all 
$i\in[|u|]$ and $P(x[|x|])=P(y[1])=0$ (if $x$ or $y$ are empty the constraint does not have to be fulfilled). 
For the property $P_\ta$ defined by $P_{\ta}(x)=1$ iff $x=\ta$ for $x\in\Sigma$, the
word $\ta\tb\ta\ta\ta\tb\ta\ta\tb\tb$ has 3 $P_{\ta}$-blocks (or short three $\ta$-blocks).

In the following, we give the main definitions on $k$-locality, following~\cite{localpatterns}. 

\begin{definition}\label{basedefs}
	Let $\overline{\Sigma}=\{\overline{x}\mid x\in \Sigma\}$ be the set of \emph{marked letters}. For a word $w\in\Sigma^{\ast}$, a {\em marking sequence} $e$ of the 
	letters occurring in $w$, is an enumeration $(x_1,x_2,\ldots, x_{|\letters(w)|})$ of $\letters(w)$. We say that $\ta_i\leq_{e}\ta_j$ if $\ta_i$ occurs before $\ta_j$ in $e$, for $i,j\in[|\letters(w)|]$. The enumeration obeying the total order of the alphabet is called the {\em canonical marking sequence} $e_{\Sigma}$. A letter $x_i$ is called {\em marked at stage $k\in\N$}  if $i\leq k$.
	Moreover, we define 
	$w_k$, {\em the marked version of $w$ at stage $k$}, as the word obtained from $w$ 
	by replacing all $x_i$ with $i\leq k$ by $\overline{x_i}$. A factor of $w_k$ is a \emph{marked block} if the defining property of the block is that it contains only elements from $\overline{\Sigma}$. 
	The locality of a word $w$ w.r.t. a marking 
	sequence $e$ ($\loc_{e}(w)$) is the maximal number of marked blocks 
	that occurred during the marking process.
\end{definition}

In the context of Definition \ref{basedefs}, $w_{|\letters(w)|}$ is always completely marked.
Using the idea of a marking sequence, we define the $k$-locality of a word.

\begin{definition}\label{def1}
	A word $w\in\Sigma^{\ast}$ is {\em $k$-local} for $k\in\N_0$ if 
	there exists a marking sequence 
	$(x_1,\ldots,x_{|\letters(w)|})$ of $\letters(w)$, such that, for all $i\leq |\letters(w)|$ we have that 
	$w_i$ at stage $i$, has at most $k$ marked blocks. A word is called {\em strictly} 
	$k$-local if it is $k$-local but not $(k-1)$-local. 
\end{definition}

When inspecting the respective sequences for marking words w.r.t. a given marking sequence the authors in \cite{fleischmann} introduced the notion of blocksequences.

\begin{definition}
	Let $w\in\Sigma^{\ast}$ and $e=(y_1,\dots,y_{\ell})$ be a marking sequence. The {\em blocksequence} $\beta_{e}(w)$
	is the sequence $(b_1,\dots,b_{\ell})$ over $\N$ such that in $e$'s \nth{$i$} 
	stage on marking $w$, $b_i$ blocks are marked, for all $i\in[\ell]$. 
\end{definition}

Due to the type of extension of a marked block the authors also introduced the notion of  neighbours, joins and singletons to characterise the relative position of the upcoming marked letter and the already marked blocks. We only recall the definition of neighbours here.

\begin{definition}\label{rules}
	Let $e=(y_1,\dots,y_{\ell})$ be a marking sequence of $w\in\Sigma^{\ast}$. 
	At stage $i\in[\ell]$, an occurrence of $y_i$ is said to be a {\em neighbour} if there exist $u_1\in\overline{\Sigma}^{+}$, 
	$u_2\in\Sigma^{+}$ and $v_1,v_2\in(\Sigma\cup\overline{\Sigma})^{\ast}$ with 
	$w_i=v_1u_1y_iu_2v_2$, $w_i=v_1u_2y_iu_1v_2$, $w_i=v_1u_1y_i$, or $w_i=y_iu_1v_1$.
	A marking sequence $e$ is called {\em neighbourless} for a word $w\in\Sigma^{\ast}$ if in any stage while marking $w$ with $e$ no neighbour occurrences exist. A word $w\in\Sigma^{\ast}$ is called {\em neighbourless} if there exists a neighbourless marking sequence $e$ for $w$.
\end{definition}

\begin{remark}\label{lookatcondensed}
	Notice that neighboured letters that are the same do not have any impact on a words locality since they are marked together. Formally, if $w = x_1^{k_1} x_2^{k_2} \cdots x_\ell^{k_\ell}\in \Sigma^*$ with $k_i, \ell \in \N, i\in [\ell]$ and $x_i \in \Sigma$, then the \emph{condensed form} of $w$ is defined by $\operatorname{cond}(w) = x_1 \cdots x_\ell$.
	We have that $w$ is strictly $k$-local iff $\cond(w)$ is strictly $k$-local since powers of letters are marked at once and for one block each, i.e., the same marking sequence applies for $w$ and $\cond(w)$.
\end{remark}

\subsection{Graph Theory}

\begin{definition}
	In a graph $G=(V,E)$ the \emph{contraction of edge} $e = \{v,v'\} \in E$ is the replacement of $v$ and $v'$ with a new vertex such that edges incident to this new vertex are the edges other than $e$ that were incident with $v$ or $v'$.
	A graph $H$ is called an \emph{induced minor} of a graph $G$ if it can be obtained from an induced subgraph of $G$ by contracting edges.
\end{definition}

\begin{definition}
	For $n \in \N_{\geq 3}$, a \emph{cycle graph} is a graph on $n$ vertices and $n$ edges forming a cycle of length $n$.
\end{definition}

\begin{definition}
	For $n \in \N$, the \emph{complete graph $K_n$} on $n$ vertices is defined by $K_n = (\{1,\ldots,n\}, \binom{\{1,\ldots,n\}}{2})$ and the edgeless graph on $n$ vertices by $I_n = (\{1,\ldots,n\}, \emptyset)$.
\end{definition}

\begin{definition}
	For $n \in \N$, the \emph{star graph} is defined by $(\{1,\dots,n\},\{\{1,i\}\mid i > 1\})$.
\end{definition}

\begin{definition}
	For $n \in \N$, the \emph{wheel graph} $W_n$ on $n$ vertices constructed by connecting a single vertex to every vertex in an $(n-1)$-cycle.
\end{definition}

\begin{definition}
	For $k \in \N$, the \emph{friendship graph} $F_k$ is constructed of $k$ copies of a $3$-cycle joined in a common vertex.
\end{definition}

\begin{definition}
	\emph{Web graphs} $W_{\ell,r}$, $\ell,r \in \N, \ell >2$ are graphs consisting of $r$ equally rotated copies of the cycle graph $C_n$ connected via the corresponding vertices.
\end{definition}

\begin{definition}
	\emph{Hypercubes} $Q_m$ are graphs formed by the vertices and edges from an $m$-dimensional hypercube.
\end{definition}

\begin{definition}
	For $d_1,d_2 \in \N$, the $d_1 \times d_2$ \emph{knight's graph} $KG$ is a graph on $d_1 d_2$ vertices in which each vertex represents a square in an $d_1 \times d_2$ chessboard, and each edge corresponds to a legal move by a knight.
\end{definition}

\begin{definition}
	\emph{Gear graphs} $G_d$ are wheel graphs $W_d$ with a vertex added between each pair of adjacent vertices in the outer cycle for $d \in \N$.
\end{definition}

\begin{definition}
	For $d \in \N$ , \emph{sunflower graphs} $S_d$ consist of a wheel graph $W_d$ such that for every two vertices of the outer cycle a vertex is added such that those three vertices form a $C_3$.
\end{definition}

\begin{definition}
	For $d \in \N$, the \emph{helm graphs} $H_d$ is a graph consisting of a wheel graph $W_d$ by adding a pendant edge at each node of the cycle of $W_d$.
\end{definition}

\begin{definition}
	For $d \in \N_{\geq 3}$, \emph{$d$-crown graphs} are graphs of the following form $(\{x_1, \cdots, x_d, y_1, \ldots, y_d\},\{(x_i,y_j) \mid i,j \in [d], i\neq j\})$.
\end{definition}

\begin{definition}
	\emph{Complete bipartite graphs} $G = (V_1 \dot{\cup} V_2,E)$ are graphs where for all $\{e_1,e_2\} \in E$ we have that $e_1 \in V_1$ and $e_2 \in V_2$.
\end{definition}

\begin{definition}
	For $d\in \N_{\geq 2}$, \emph{complete $n$-partite graph} $G = (V_1 \dot{\cup} V_2 \dot{\cup} \ldots \dot{\cup} V_d,E)$ are graphs where for all $i \in [d]$ we have that $v_1 \in V_i$ and $v_2 \in V_i$ then $\{v_1,v_2\} \notin E$.
\end{definition}

\begin{definition}
	A graph $G=(V,E)$ is \emph{$k$-regular} for $k\in \N$ if every vertex $v \in V$ has exactly $k$ outgoing edges.
\end{definition}

\begin{definition}
	\emph{Flower snarks} $J_d$ for $d \in 2\N +1$ are snarks that are non-planar and non-Hamiltonian \cite{Isaacs1975InfiniteFO,flowersnarks}. The vertex set of $J_d$ consists of $4d$ vertices $V_{J_d} = \{v_1, \ldots, v_d\} \cup \{u_1^1, u_1^2, u_1^3, \ldots , u_d^1, u_d^2, u_d^3\}$. The graph is comprised of a cycle $u_1^1 \cdots u_d^1$ of length $d$ and a cycle $u_1^2 \cdots u_d^2 u_1^3 \cdots u_d^3$ of length
	$2d$. Further, every vertex $v_i$ is adjacent to $u_i^1, u_i^2$ and $u_i^3$.
\end{definition}

\end{document}